\documentclass[11pt]{amsart}
\usepackage{amssymb}
\usepackage{amsmath}

\numberwithin{equation}{section}

\newtheorem{thm}{Theorem}[section]
\newtheorem{lem}[thm]{Lemma}
\newtheorem{cor}[thm]{Corollary}

\newtheorem{defin}[thm]{Definition}
\newtheorem{rem}[thm]{Remark}

\def\id{\mathop{\rm id}}

\def\id{{\bf 1}\!\!{\rm I}}

\newcommand{\be}{\begin{equation}}
\newcommand{\ee}{\end{equation}}
\newcommand{\bes}{\begin{equation*}}
\newcommand{\ees}{\end{equation*}}

\newcommand{\bea}{\begin{eqnarray}}
\newcommand{\eea}{\end{eqnarray}}
\newcommand{\ba}{\begin{array}}
\newcommand{\ea}{\end{array}}
\newcommand{\bc}{\begin{center}}
\newcommand{\ec}{\end{center}}

\def\ba{{\mathbb A}}

\def\bc{{\mathbb C}}

\def\bn{{\mathbb N}}

\def\br{{\mathbb R}}

\def\l{\lambda}

\def\i{\varepsilon}

\def\f{\varphi}

\def\N{\mathbb{N}}

\def\a{\alpha}

\def\b{\beta}
\def\m{\mu}

\def\d{\delta}
\def\g{\gamma}

\def\ck{\mathcal{K}}

\begin{document}
%\Large
\title[ergodicity coefficient and weak ergodicity]
{ergodic properties of nonhomogeneous Markov chains defined on
ordered Banach spaces with a base}

\author{Farrukh Mukhamedov}
\address{Farrukh Mukhamedov\\
 Department of Computational \& Theoretical Sciences\\
Faculty of Science, International Islamic University Malaysia\\
P.O. Box, 141, 25710, Kuantan\\
Pahang, Malaysia} \email{{\tt far75m@yandex.ru}, {\tt
farrukh\_m@iium.edu.my}}

\begin{abstract}
It is known that the Dobrushin's ergodicity coefficient is one of
the effective tools to study a behavior of non-homogeneous Markov
chains. In the present paper, we define such an ergodicity
coefficient of a positive mapping defined on ordered Banach space
with a base (OBSB), and study its properties. In terms of this
coefficient we prove the equivalence uniform and weak ergodicities
of homogeneous Markov chains. This extends earlier results obtained
in case of von Neumann algebras. Such a result allowed to establish
a category theorem for uniformly ergodic Markov operators. We find
necessary and sufficient conditions for the weak ergodicity of
nonhomogeneous discrete Markov chains (NDMC). It is also studied
$L$-weak ergodicity of NDMC defined on OBSB. We establish that the
chain satisfies $L$-weak ergodicity if and only if it satisfies a
modified Doeblin's condition ($\frak{D}_1$-condition). Moreover,
several nontrivial examples of NDMC which satisfy the
$\frak{D}_1$-condition are provided. \\[2mm]
\end{abstract}

\maketitle

%\large

\section{Introduction}

It is well-known that the transition probabilities play crucial role
in the theory of Markov processes. Using such probabilities one can
define a linear operator, which is called {\it Markov operator}
acting on $L^1$-spaces. In the study of several ergodic properties
of the Markov process, the investigation of asymptotical behaviors
of iterations of the Markov operators plays an important role (see
\cite{B,K,Sz}). In last years, the study of quantum dynamical
systems has had an impetuous growth due to natural applications of
quantum dynamical systems to various fields of mathematics and
physics. Therefore, ergodic properties of Markov operators were also
investigated by many authors. The reader is referred e.g. to
\cite{AH,FR1,FR2,J1,NSZ,S} for further details relative to some
differences between the classical and the quantum situations. On the
other hand, to study several properties of physical and
probabilistic processes in abstract framework is convenient and
important. Some applications of this scheme in quantum information
have been discussed in \cite{RKW}. One can see that the classical
and quantum cases confine to this scheme. In this abstract scheme
one considers norm ordered spaces (see \cite{Alf}). Mostly, in those
investigations homogeneous Markov processes were considered (see
\cite{BK1,BK2,EW1,J1,L,RKW,SZ}). We should stress that a limited
number of papers (see \cite{Ber}) were devoted to the investigations
of ergodic properties of nonhomogeneous Markov processes in the
mentioned scheme. Therefore, in this paper we study more deeper the
nonhomogeneous Markov processes in the abstract framework.

Note that due to nonhomogenety of the process the investigations of
limiting behavior of such processes becomes very complicated. One of
the important tools in this study is so-called the Dobrushin's
ergodicity coefficient \cite{D,MI,IS,T}. A lot of papers were
devoted to the investigation of ergodicity of nonhomogeneous Markov
chains (see, for example \cite{D}-\cite{JI},\cite{Se,ZI}).

As we said before that the Dobrushin's ergodicity coefficient is one
of the effective tools to study a behavior of such products (see
\cite{IS} for review). Therefore, in section 3, we will define such
an ergodicity coefficient $\d(T)$ of a positive mapping $T$ defined
on ordered Banach space with a base, and study its properties. Note
that such a coefficient has been independently defined in \cite{GQ}.
But, its useful properties were not provided. In terms of this
coefficient we prove the equivalence uniform and weak ergodicities
of homogeneous Markov chain. This extends earlier results obtained
in \cite{BK1,M1} in case of von Neumann algebras. Such a result
allowed to establish a category theorem for uniformly ergodic Markov
operators. We point out that in the setting of general stochastic
operators on classical $L^1$-spaces similar kind of result was
originally discussed in \cite{Iw} and further refinements due to W.
Bartoszek \cite{B0}. We remark that initially the question on the
geometric structure of the set of uniformly ergodic operators its
size and category is rooted in \cite{Hal}. Further in section 4, we
find necessary and sufficient conditions for the weak ergodicity of
nonhomogeneous discrete Markov chains (NDMC), which extend the
results of \cite{MI,M1} to an abstract scheme. Note that in
\cite{DP} similar conditions were found for classical nonhomogeneous
Markov processes to satisfy weak ergodicity. In section 5, we study
$L$-weak ergodicity of NDMC defined on ordered Banach spaces with a
base. It is established that the chain satisfies $L$-weak ergodicity
if and only if it satisfies a modified Doeblin's condition
($\frak{D}_1$-condition). Note that in case of classical
$L^1$-spaces a similar kind of result has been proved in \cite{M2}.
In the final section 6, we provide several nontrivial examples of
NDMC which satisfy the $\frak{D}_1$-condition.

\section{Preliminaries}

Let $X$ be an ordered vector space with a cone $X_+=\{x\in X: \
x\geq 0\}$. A subset $\ck$ is called a {\it base} for $X$, if one
has $\ck=\{x\in X:\ f(x)=1\}$ for some strictly positive (i.e.
$f(x)>0$ at $x>0$) linear functional $f$ on $X$. An ordered vector
space $X$ with generating cone $X_+$ (i.e. $X=X_+-X_+$) and a fixed
base $\ck$, defined by a functional $f$, is called {\it an ordered
vector space with a base} \cite{Alf}. In what follows, we denote it
as $(X,X_+,\ck,f)$. Let $U$ be the convex hull of the set
$\ck\cup-\ck$, and let
$$
\|x\|_{\ck}=\inf\{\l\in\br_+:\ x\in\l U\}.
$$
Then one can see that $\|\cdot\|_{\ck}$ is a seminorm on $X$.
Moreover, one has $\ck=\{x\in X_+: \ \|x\|_{\ck}=1\}$,
$f(x)=\|x\|_{\ck}$ for $x\in X_+$. If the set $U$ is linearly
bounded (i.e. for any line $\ell$ the intersection $\ell\cap U$ is a
bounded set), then $\|\cdot\|_{\ck}$ is a norm, and in this case
$(X,X_+,\ck,f)$ is called {\it an ordered norm space with a base}.
When  $X$ is complete with respect to the norm $\|\cdot\|_{\ck}$ and
the cone $X_+$ is closed, then $(X,X_+,\ck,f)$ is called {\it an
ordered Banach space with a base (OBSB)}. In the sequel, for the
sake of simplicity instead of $\|\cdot\|_{\ck}$ we will use usual
notation $\|\cdot\|$.

Let us provide some examples of OBSB.

\begin{itemize}
\item[1.] Let $E$ be a order-unit normed space. Then the conjugate space
$E^*$ is OBSB (see \cite{Alf}).

\item[2.] Let $M$ be a von Neumann algebra. Let $M_{h,*}$ be the
Hermitian part of the predual space $M_*$ of $M$. As a base $\ck$ we
define the set of normal states of $M$. Then
$(M_{h,*},M_{*,+},\ck,\id)$ is a OBSB, where $M_{*,+}$ is the set of
all positive functionals taken from $M_*$, and $\id$ is the unit in
$M$.

\item[3.] Let $X=\ell_p$, $1<p<\infty$. Define
$$
X_+=\bigg\{\mathbf{x}=(x_0,x_1,\dots,x_n,\dots)\in\ell_p: \ x_0\geq
\bigg(\sum_{i=1}^\infty|x_i|^p\bigg)^{1/p}\bigg\}
$$
and $f_0(\mathbf{x})=x_0$. Then $f_0$ is a strictly positive linear
functional. In this case, we define $\ck=\{x\in X_+: \
f_0(\mathbf{x})=1\}$. Then one can see that $(X,X_+,\ck,f_0)$ is a
OBSB. Note that the norm $\|\cdot\|_{\ck}$ is equivalent to the
usual $\ell_p$-norm.
\end{itemize}

Let $(X,X_+,\ck,f)$ be an OBSB. It is well-known (see
\cite[Proposition II.1.14]{Alf}) that every element $x$ of OBSB
admits a decomposition $x=y-z$, where $y,z\geq 0$ and
$\|x\|=\|y\|+\|z\|$.

Let $(X,X_+,\ck,f)$ be an OBSB. A linear operator $T:X\to X$ is
called positive, if $Tx\geq 0$ whenever $x\geq 0$. A positive linear
operator $T:X\to X$ is called {\it Markov}, if $T(\ck)\subset\ck$.
It is clear that $\|T\|=1$ and its adjoint mapping $T^*: X^*\to X^*$
acts in ordered Banach space $X^*$ with unit $f$, and moreover, one
has $T^*f=f$. Note that in case of $X=\br^n$, $X_+=\br_+^n$ and
$\ck=\{(x_i)\in\br^n: \ x_i\geq 0, \ \sum_{i=1}^n x_i=1\}$, then for
any Markov operator $T$ acting on $\br^n$, the operator $T^*$
coincides with usual stochastic matrix. Now for each $y\in X$ we
define a linear operator $T_y: X\to X$ by $T_y(x)=f(x)y$.
 Recall that a family of Markov operators
$\{T^{m,n}: X\to X\}$ ($m\leq n$, $m,n\in\N$) is called a
\textit{nonhomogeneous discrete Markov
 chain (NDMC)} if
$$
T^{m,n}=T^{k,n}T^{m,k}
$$
for every $m\leq k\leq n$. A NDMC  $\{T^{m,n}\}$ is called {\it
uniformly asymptotically stable} or {\it uniformly ergodic} if there
exist an element $y_0\in X$ such that
$$
\lim_{n\to\infty}\|T^{m,n}-T_{y_0}\|=0
$$
for any $m\geq 0$.

Recall that if for a NDMC $\{T^{k,m}\}$ one has
$T^{k,m}=(T^{0,1})^{m-k}$, then such a chain becomes {\it
homogeneous}. In what follows, by $\{T^n\}$ we denote homogeneous
Markov chain, where $T:=T^{0,1}$.

\section{Dobrushin ergodicity coefficient}

Let $(X,X_+,\ck,f)$ be an OBSB  and $T:X\to X$ be a linear bounded
operator. Define
$$
N=\{x\in X: \ f(x)=0\},
$$
\be \label{db} \d(T)=\sup_{x\in N,\ x\neq 0}\frac{\|Tx\|}{\|x\|}.
\ee

The magnitude $\d(T)$ is called the \textit{Dobrushin ergodicity
coefficient} of $T$.

\begin{rem} We note that if $X^*$ is a commutative algebra, the notion of
the Dobrushin ergodicity coefficient was studied in
\cite{C},\cite{D},\cite{ZZ}. In a noncommutative setting, i.e. when
$X^*$ is a von Neumann algebra, such a notion was introduced in
\cite{M,MTA}. We should stress that such a coefficient has been
independently defined in \cite{GQ}. Furthermore, for particular
cases, i.e. in a noncommutative setting, such a coefficient
explicitly has been calculated for quantum channels (i.e. completely
positive maps). But, its useful properties were not provided. Below,
we will prove several important properties of the defined
coefficient.
\end{rem}

Before formulating the main result of this section we need the
following auxiliary result.

\begin{lem}\label{3.2} For every $x,y\in X$ such that $x-y\in N$ there
exist $u,v\in \ck$ with
$$
x-y=\frac{\|x-y\|}{2}(u-v).
$$
\end{lem}

\begin{proof} Denote $z=x-y$. Then due to \cite[Proposition
II.1.14]{Alf} one can find $a,b\in X_+$ such that $z=a-b$ and
$\|z\|=\|a\|+\|b\|$. One can see that $f(z)=0$, which yields
$f(a)=f(b)$. From $f(a)=\|a\|$, $f(b)=\|b\|$ we get $\|a\|=\|b\|$.
Hence, $\|a\|=\|z\|/2$. By putting
$$
u=\frac{a}{\|a\|}, \ \ v=\frac{b}{\|b\|}
$$
we obtain the required assertion.
\end{proof}

 The next result establishes several properties of the Dobrushin
ergodicity coefficient. Note that when $X^*$ is a von Neumann
algebra, similar properties were studied in \cite{M,IS}.

\begin{thm}\label{d-prp} Let $(X,X_+,\ck,f)$ be an OBSB, and $T,S: X\to X$ be Markov operators. The following assertions hold:
\begin{enumerate}
\item[(i)] $0\leq \d(T)\leq 1$;

\item[(ii)] $|\d(T)-\d(S)|\leq\d(T-S)\leq \|T-S\|$;

\item[(iii)] $\d(TS)\leq\d(T)\d(S)$;

\item[(iv)] if $H: X\to X$ is a linear bounded
operator such that $H^*(f)=0$, then $\|TH\|\leq\d(T)\|H\|$;

\item[(v)] one has
\be \label{db2} \d(T)=\frac{1}{2}\sup_{u,v\in \ck}\|Tu-Tv\|. \ee

\item[(vi)] if $\d(T)=0$, then there is $y_0\in X_{+}$ such that
$T=T_{y_0}$.
\end{enumerate}
\end{thm}

\begin{proof} (i) is obvious. Let us prove (ii). From \eqref{db} we immediately
find that $\d(T-S)\leq \|T-S\|$. Now let us establish the first
inequality. Without loss of generality, we may assume that
$\d(T)\geq\d(S)$. For an arbitrary $\i>0$ from \eqref{db} one finds
$x_\i\in N$ with $\|x_\i\|=1$ such that
$$
\d(T)\leq\|Tx_\i\|+\i.
$$
Then we get
\begin{eqnarray*}
\d(T)-\d(S)&\leq&\|Tx_\i\|+\i-\sup_{x\in N,\|x\|=1}\|Sx\|\\[2mm]
&\leq & \|Tx_\i\|-\|S\f_\i\|+\i\\
&\leq&\|(T-S)x_\i\|+\i\\
&\leq&\sup_{x\in N: \ \|x\|=1}\|(T-S)x\|+\i\\
&=&\d(T-S)+\i
\end{eqnarray*}
which implies the assertion.

(iii). Let $x\in N$, then the Markovianity of $S$ with Lemma
\ref{3.2} implies $f(Sx)=0$, hence $Sx\in N$. Due to \eqref{db} one
finds
\begin{eqnarray*}
\|TSx\|\leq\d(T)\|Sx\|\leq \d(T)\d(S)\|\f\|
\end{eqnarray*}
which yields $\d(TS)\leq\d(T)\d(S)$.

(iv). Let $H:X\to X$ be a linear bounded operator such that
$K^*(f)=0$. Then for every $x\in X$ one gets $Hx\in N$. Therefore,
\begin{eqnarray*}
\|THx\|\leq\d(T)\|Hx\|\leq\|H\|\d(T)\|x\|
\end{eqnarray*}
which yields the assertion.

 (v).  For $x\in N$, $x\neq 0$ using
Lemma \ref{3.2} we have \bea \frac{\|Tx\|}{\|x\|}=
\frac{\frac{\|x\|}{2}\|T(u-v)\|}{\|x\|}=\frac{\|Tu-Tv\|}{2}.\nonumber
\eea The equality (\ref{db}) together with the last equality imply
the desired one (\ref{db2}).

(vi). Let $\d(T)=0$, then from \eqref{db2} one gets $Tu=Tv$ for all
$u,v\in \ck$. Therefore, let us denote $y_0:=Tu$. It is clear that
$y_0\in \ck$. Moreover, $Ty_0=y_0$. Let $x\in X_{+}$, then noting
$\|x\|=f(x)$ we find
$$Tx=\|x\|T\bigg(\frac{x}{\|x\|}\bigg)=f(x)y_0.
$$

If $x\in X$, then $x=y-z$, where $y,z\geq 0$. Therefore,
$$
T(x)=T(y)-T(z)=f(y)y_0-f(z)y_0=f(x)y_0.
$$
This completes the proof.
\end{proof}

\begin{rem} We stress that similar kind of equality like (v) in
Theorem \ref{d-prp} has been proved in \cite{RKW} in case of finite
dimensional OBSB setting (see also \cite{GQ}).
\end{rem}

First we recall that a NDMC $\{T^{k,n}\}$ defined on $X$ is
\textit{weakly ergodic} if for each $k\in\N\cup\{0\}$ one has
$$
\lim_{n\to\infty}\sup_{x,y\in \ck}\|T^{k,n}x-T^{k,n}y\|=0.
$$
Note that taking into account Theorem \ref{d-prp}(v) we obtain that
the weak ergodicity is equivalent to the condition $\d(T^{k,n})\to
0$ as $n\to\infty$.

\begin{thm}\label{3.4} Let $(X,X_+,\ck,f)$ be an OBSB and $\{T^n\}$ be a
discrete homogeneous Markov chain on $X$.  The following conditions
are equivalent:
\begin{enumerate}
      \item[(i)] the chain $\{T^n\}$ is weakly ergodic;
\item[(ii)] there exists $\rho\in[0,1)$ and $n_0\in\N$ such that
   $\d(T^{n_0})\leq \rho$;
   \item[(iii)] the chain $\{T^n\}$ is uniformly ergodic. Moreover,
   there are positive constants $C,\a$, $n_0\in\bn$ and $y_0\in\ck$ such that
\begin{equation}\label{exp-11}
\|T^n-T_{y_0}\|\leq C\cdot e^{-\a n}, \ \ \forall n\geq n_0.
\end{equation}
\end{enumerate}
\end{thm}

\begin{proof} The implications (i) $\Rightarrow$ (ii) and (iii) $\Rightarrow$ (i) are obvious.
Therefore, it is enough to prove the implication (ii) $\Rightarrow$
(iii). Let $\rho\in[0,1)$ and $n_0\in\N$ such that $\d(T^{n_0})\leq
\rho$. Now from (iii) and (i) of Theorem \ref{d-prp} one gets
\begin{equation}\label{dn}
\d(T^n)\leq \rho^{[n/n_0]}\to 0 \ \ \textrm{as} \ \ n\to\infty,
\end{equation}
where $[a]$ stands for the integer part of $a$.

Let us show that  $\{T^n\}$ is a Cauchy sequence w.r.t. to the norm.
Indeed, using (iv) of Theorem \ref{d-prp} and \eqref{dn} we have
\begin{eqnarray}\label{dn1}
\|T^n-T^{n+m}\|&=&\|T^{n-1}(T-T^{m+1})\|\nonumber \\[2mm]
&\leq&\d(T^{n-1})\|T-T^{m+1}\|\to 0 \ \ \textrm{as} \ \ n\to\infty.
\end{eqnarray}
Hence, there is a Markov operator $Q$ such that $\|T^n-Q\|\to 0$.
Let us show that $Q=T_{y_0}$, for some $y_0\in X_{+}$. According to
(vi) of Theorem \ref{d-prp} it is enough to establish $\d(Q)=0$.

So, using (ii) of Theorem \ref{d-prp} we have
\begin{eqnarray*}
|\d(T^n)-\d(Q)|\leq \|T^n-Q\|.
\end{eqnarray*}
Now passing to the limit $n\to\infty$ at the last inequality and
taking into account \eqref{dn}, one gets $\d(Q)=0$, which is the
desired assertion. From \eqref{dn1},\eqref{dn} and $Q=T_{y_0}$ we
immediately get \eqref{exp-11}.

\end{proof}

\begin{rem} Note that the proved theorem is a abstract version Bartoszek's
result \cite{B}. Moreover, the proved theorem extends earlier
results obtained in \cite{BK1} for $X=\mathcal{C}_1$, where
$\mathcal{C}_1$ is the Schatten class 1. A similar result has been
obtained in \cite{M} in case is von Neuamann algebra. In \cite{GQ}
uniform ergodicity of $\{T^n\}$ defined on order-unit space, under a
stronger condition (i.e. $\d(T)<1$) than provided one, has been
proved. In Theorem \ref{3.4} the condition $\rho<1$ is crucial,
otherwise the statement of the theorem fails (see \cite[Remark
3.3]{M}).
\end{rem}

Let us denote by $\Sigma(X)$ the set of all Markov operators defined
on $X$ By $\Sigma(X)_{ue}$ we denote the set of all Markov operators
$T$ such that the corresponding discrete Markov chain $\{T^n\}$ is
uniformly ergodic.

\begin{thm}\label{Norm-M} Let $(X,X_+,\ck,f)$ be an OBSB. Then the set $\Sigma(X)_{ue}$ is
a norm dense and open subset of $\Sigma(X)$.
\end{thm}

\begin{proof} Take an arbitrary $T\in\Sigma(X)$ and $0<\i<2$. Given
$\phi\in \ck$ let us denote
$$
T^{(\i)}=\bigg(1-\frac{\i}{2}\bigg)T+\frac{\i}{2} T_\phi.
$$
It is clear that $T^{(\i)}\in\Sigma(X)$ and $\|T-T^{(\i)}\|<\i$. Now
we show that $T^{(\i)}\in\Sigma(X)_{ue}$. Indeed, by Lemma
\ref{3.2}, if $x-y\in N$, we get
\begin{eqnarray*}
\|T^{(\i)}(x-y)\|&=&\frac{\|x-y\|}{2}\|T^{(\i)}(u-v)\| \\[2mm]
&=&\frac{\|x-y\|_1}{2}\bigg\|\bigg(1-\frac{\i}{2}\bigg)T(u-v)+
\frac{\i}{2}T_\phi(u-v)\bigg\|\\[2mm]
&=&\frac{\|x-y\|}{2}\bigg\|\bigg(1-\frac{\i}{2}\bigg)T(u-v)\bigg\|\\[2mm]
&\leq&\bigg(1-\frac{\i}{2}\bigg)\|x-y\|
\end{eqnarray*}
which implies $\d(T_\i)\leq 1-\frac{\i}{2}$. Here $u,v\in \ck$.
Hence, due to Theorem \ref{3.4} we infer that
$T_\i\in\Sigma(X)_{ue}$.

Now let us show that $\Sigma(X)_{ue}$ is a norm open set. First we
establish that for each $n\in\N$ the set
$$
\Sigma(X)_{ue,n}=\bigg\{T\in \Sigma(X):  \ \ \d(T^n)< 1\bigg\}
$$
is open. Indeed, take any $T\in\Sigma(X)_{ue,n}$, then
$\a:=\d(T^n)<1$. Choose $0<\b<1$ such that $\a+\b<1$. Then for any
$H\in\Sigma(X)$ with $\|H-T\|<\b/n$ by using (ii) Theorem
\ref{d-prp} we find
\begin{eqnarray*}
|\d(H^n)-\d(T^n)|&\leq&\|H^n-T^n\|\\[2mm]
&\leq& \|H^{n-1}(H-T)\|+\|(H^{n-1}-T^{n-1})T\|\\[2mm]
&\leq& \|H-T\|+\|H^{n-1}-T^{n-1}\|\\[2mm]
&\cdots&\\
&\leq& n\|H-T\|<\b.
\end{eqnarray*}
Hence, the last inequality yields that $\d(H^n)<\d(T^n)+\b<1$, i.e.
$H\in\Sigma(X)_{ue,n}$.

Now from the equality
$$
\Sigma(X)_{ue}=\bigcup_{n\in\N}\Sigma(X)_{ue,n}
$$
we obtain that $\Sigma(X)_{ue}$ is open. This completes the proof.
\end{proof}

\begin{rem} Note that a similar result has been proved in Theorem 2.4
\cite{BK2} for $X=\mathcal{C}_1$, where $\mathcal{C}_1$ is the
Schatten class 1. So, the proved theorem extends Theorem 2.4 for
general abstract spaces. We point out that in the setting of general
stochastic operators on classical $L^1$-spaces such a result was
originally discussed in \cite{Iw} and further refinements due to W.
Bartoszek \cite{B0}. The question on the geometric structure of the
set of uniformly ergodic operators its size and category was
initiated in \cite{Hal}.
\end{rem}

\section{Weak ergodicity of nonhomogeneous Markov chains}

In this section we are going to illustrate the usefulness of the
Dobrushin ergodicity coefficient in study of weak ergodicity of
nonhomogeneous discrete Markov chains defined on OBSB.

Let $(X,X_+,\ck,f)$ be an OBSB. We say that a NDMC $\{T^{k,n}\}$
defied on $X$ satisfies \textit{condition $\frak{D}$} if for each
$k$ there exist an element $z_k\in \ck$, a constant $\l_k\in(0,2]$,
an integer $n_k\in\N$, and for every $x\in \ck$, one can find
$\f_{k,x}\in X_{+}$ with $\sup\limits_{x}\|\f_{k,x}\|\leq
\frac{\l_k}{4}$ such that
\begin{equation}\label{Dk}
T^{k,n_k}x+\f_{k,x}\geq\l_k z_k,
\end{equation}
and
\begin{equation}\label{Lam1}
\liminf_{k\to\infty}\l_{k}>0.
\end{equation}

Note that the given condition is an abstract analogue of the
well-known \cite{Se} Doeblin's Condition.

\begin{thm}\label{N2}
Let $(X,X_+,\ck,f)$ be an OBSB  and assume that a NDMC $\{T^{k,n}\}$
defined on $X$. Then the following conditions are equivalent:
\begin{enumerate}
\item[(i)] the chain $\{T^{k,n}\}$ satisfies condition $\frak{D}$;

\item[(ii)] the chain $\{T^{k,n}\}$ is weak ergodic.
\end{enumerate}
\end{thm}

\begin{proof} (i) $\Rightarrow$ (ii). Fix $k\in\N\cup\{0\}$, then one finds $z_k\in \ck$.
According to condition $\frak{D}$ for any two elements $x,y\in\ck$,
there exist $\l_k\in[0,2]$, $n_k\in\N$ and $\f_{k,x},\f_{k,y}\in
X_{+}$ with $\|\f_{k,x}\|\leq \frac{\l_k}{4}$, $\|\f_{k,y}\|\leq
\frac{\l_k}{4}$  such that
\begin{equation}\label{N22}
T^{k,k+n_k}x+\f_{k,x}\geq\l_k z_k, \ \ T^{k,k+n_k}y+\f_{k,y}\geq\l_k
z_k.
\end{equation}

Now denote $\f_k=\f_{k,x}+\f_{k,y}$, then we have
\begin{equation}\label{N23}
\|\f_k\|\leq\frac{\l_k}{2}.
\end{equation}
From \eqref{N22} one finds
\begin{eqnarray}\label{N24}
T^{k,k+n_k}x+\f_{k}\geq T^{k,k+n_k}x+\f_{k,x}\geq\l_k z_k.
\end{eqnarray}
Similarly,
\begin{eqnarray}\label{N25}
T^{k,k+n_k}y+\f_{k}\geq \l_k z_k.
\end{eqnarray}

Therefore, using Markovianity of $T^{k,n}$, and inequality
\eqref{N24} with \eqref{N23} implies
\begin{eqnarray*}
\|T^{k,k+n_k}x+\f_{k}-\l_k z_k\|&=&f(T^{k,k+n_k}x+\f_{k}-\l_k z_k)\\[2mm]
&=&1-(\underbrace{\l_k-f(\f_{k})}_{c_1})\\[2mm]
&=&1-c_1\leq 1-\frac{\l_k}{2}.
\end{eqnarray*}
By the same argument and using \eqref{N25}, we find
\begin{eqnarray*}
\|T^{k,k+n_k}y+\f_{k}-\l_k z_k\|=1-c_1\leq 1-\frac{\l_k}{2}.
\end{eqnarray*}

Let us denote
\begin{eqnarray*}
x_1=\frac{1}{1-c_1}(T^{k,k+n_k}x+\f_{k}-\l_k z_k),\\[2mm]
y_1=\frac{1}{1-c_1}(T^{k,k+n_k}y+\f_{k}-\l_k z_k).
\end{eqnarray*}
It is clear that $x_1,y_1\in \ck$.

So, one has
\begin{equation}\label{N26}
\|T^{k,k+n_k}x-T^{k,k+n_k}y\|=(1-c_1)\|x_1-y_1\|\leq
2\bigg(1-\frac{\l_k}{2}\bigg).
\end{equation}

Hence, from \eqref{db2} and \eqref{N26} we obtain
$$
\d(T^{k,k+n_k})\leq \m_k,
$$
where $\m_k=1-\frac{\l_k}{2}$.

For $\ell_1:=k+n_k$ we again apply the given condition, then one can
find $\m_{\ell_1}$, $n_{\ell_1}$ such that
$\d(T^{\ell_1,\ell_1+n_{\ell_1}})\leq\m_{\ell_1}$. Now continuing
this procedure one finds sequences $\{\ell_j\}$ and
$\{\m_{\ell_j}\}$ such that
$$
\ell_0=k, \ \ell_1=\ell_0+n_k, \ \ell_2=\ell_1+n_{\ell_1}, \dots,
\ell_m=\ell_{m-1}+n_{\ell_{m-1}}, \dots
$$
and $\d(T^{\ell_j,\ell_{j+1}})\leq\m_{\ell_j}$.

Now for large enough $n$ one can find $L$ such that
$$L=\max\{j: \
\ell_j+n_{\ell_j}\leq n\}.$$

Then due to (iii) of Theorem \ref{d-prp} we get
\begin{eqnarray*}
\d(T^{k,n})&=&\d\big(T^{n,\ell_L}T^{\ell_{L-1},\ell_{L}}\cdots
T^{\ell_0,\ell_1}\big)\\
&\leq& \prod_{j=0}^{L-1}\d(T^{\ell_{L-j},\ell_{L-j+1}})\\
&\leq& \prod_{j=0}^{L-1}\m_{\ell_j}.
\end{eqnarray*}
Now taking into account \eqref{Lam1} one finds
$$
\limsup_{k\to\infty}\m_k<1
$$
which implies the weak ergodicity of $\{T^{k,n}\}$.

(ii)$\Rightarrow$(i). Let $\{T^{k,n}\}$ be weak ergodic. Take any
$k\in\N\cup\{0\}$, and fix some element $y_0\in\ck$. Then one gets
\begin{equation*}
\sup_{x\in\ck}\|T^{k,n}x-T^{k,n}y_0\|\to 0\ \ \ \textrm{as} \ \
n\to\infty.
\end{equation*}
Therefore, one can find $n_k\in\N$ such that
\begin{equation}\label{N27}
\sup_{x\in\ck}\|T^{k,n_k}x-T^{k,n_k}y_0\|\leq\frac{1}{4}, \ \
\end{equation}

Due to Proposition II.1.14 \cite{Alf} we can decompose
$$
T^{k,n_k}x-T^{k,n_k}y_0=(T^{k,n_k}x-T^{k,n_k}y_0)_+-(T^{k,n_k}x-T^{k,n_k}y_0)_-.
$$
Denote
$$
\f_{k,x}=(T^{k,n_k}x-T^{k,n_k}y_0)_-,
$$ It follows from \eqref{N27} that
$$
\sup_{x\in\ck}\|\f_{k,x}\|=\sup_{x\in
\ck}\|(T^{k,n_k}x-T^{k,n_k}y_0)_-\|\leq
\sup_{x\in\ck}\|T^{k,n_k}x-T^{k,n_k}y_0\|\leq\frac{1}{4}.
$$

It is clear that
\begin{eqnarray*}
T^{k,n_k}x+\f_{k,x}&=&T^{k,n_k}y_0+T^{k,n_k}x-T^{k,n_k}y_0+\f_{k,x}\\[2mm]
&=&T^{k,n_k}y_0+(T^{k,n_k}x-T^{k,n_k}y_0)_+\\
 &\geq& T^{k,n_k}y_0.
\end{eqnarray*}
By denoting $\l_k=1$ and $z_k=T^{k,n_k}y_0$, we conclude that the
process $\{T^{k,m}\}$ satisfies condition $\frak{D}$.
\end{proof}

Note that if $X$ is an $L^1$-space, then a similar result has been
proved in \cite{DP}.

\section{$L$-weak ergodicity}

Let $(X,X_+,\ck,f)$ be an OBSB and $\{T^{k,n}\}$ be a NDMC defined
on $X$.

\begin{defin}\label{d6.1} We say that $\{T^{k,n}\}$ satisfies
\begin{enumerate}
\item[(i)] the \textit{$L$-weak ergodicity} if for any $u,v\in
\ck$ and $k\in\N\cup\{0\}$ one has
\begin{equation}\label{L1wed}
\lim_{n\rightarrow\infty}\|T^{k,n}u-T^{k,n}v\|=0.
\end{equation}

\item[(ii)] the \textit{$L$-strong ergodicity} if there exists
$y_0\in \ck$ such that for every $k\in\N\cup\{0\}$ and $u\in \ck$
one has
\begin{equation}\label{L1sed}
\lim_{n\rightarrow\infty}\|T^{k,n}u-y_0\|=0.
\end{equation}
\end{enumerate}
\end{defin}

\begin{rem} It is clear that the weak ergodicity implies the
$L$-weak ergodicity. But, the reverse is not true (see \cite{M1}).
\end{rem}

\begin{rem} Note that if for each $k\geq 0$ there exists
$y_k\in \ck$ such that for every $u\in \ck$ one has
\begin{equation}\label{L1r2}
\lim_{n\rightarrow\infty}\|T^{k,n}u-y_k\|=0,
\end{equation}
then the process is the $L$-strong ergodic. Indeed, it is enough to
show that $y_0=y_k$ for all $k\geq 1$. For any $u,v\in \ck$, one has
$T^{0,n}u\to y_0$, $T^{k,n}u\to y_k$ as $n\to\infty$. From this we
conclude that $T^{k,n}(T^{0,k}u)\to y_k$ as $n\to\infty$. Now the
equality $T^{0,n}u=T^{k,n}T^{0,k}u$ implies that $y_0=y_k$.\\
\end{rem}

Now we consider two conditions for NDMC which are weaker analogous
of the condition $\frak{D}$.

Let $(X,X_+,\ck,f)$ be an OBSB and let $\{T^{k,n}\}$ be a NDMC on
$X$ and $\frak{N}\subset\ck$. We say that $\{T^{k,n}\}$ satisfies
\begin{enumerate}
\item[(a)] \textit{condition $\frak{D}_1$} on $\frak{N}$ if for each
$k$ there exist $z_k\in\ck$ and a constant $\l_k\in[0,2]$, and for
every $x,y\in \frak{N}$, one can find an integer $n_k\in\N$ and
$\f_{k,x},\f_{k,y}\in X_+$ with $\|\f_{k,x}\|\leq \l_k/4$,
$\|\f_{k,y}\|\leq \l_k/4$ such that
\begin{equation}\label{D1}
T^{k,n_k}x+\f_{k,x}\geq\l_k z_k, \ \ T^{k,n_k}y+\f_{k,y}\geq\l_k
z_k,
\end{equation}
with \eqref{Lam1}.

\item[(b)] \textit{condition
$\frak{D}_2$} on $\frak{N}$ if for each $k$ there exist $z_k\in\ck$
and a constant $\l_k\in[0,2]$, and for every $x\in \frak{N}$, one
can find a sequence $\{\f_{k,x}^{(n)}\}\subset X_+$ with
$\|\f_{k,x}^{(n)}\|\to 0$ as $n\to\infty$ such that
\begin{equation}\label{D2}
T^{k,n}x+\f_{k,x}^{(n)}\geq\l_k z_k, \ \ \textrm{for all} \ \ n\geq
k
\end{equation}
where $\{\l_k\}$ satisfies \eqref{Lam1}.
\end{enumerate}

Next theorem shows that condition $\frak{D}_2$ is stronger than
$\frak{D}_1$.

\begin{thm}\label{L1S1}
Let $(X,X_+,\ck,f)$ be an OBSB and assume that a NDMC $\{T^{k,n}\}$
defined on $X$. Then for the following statements:
\begin{enumerate}
\item[(i)] $\{T^{k,n}\}$ satisfies condition $\frak{D}_2$ on $\ck$;
\item[(ii)] $\{T^{k,n}\}$ satisfies condition $\frak{D}_2$ on a dense set $\frak{N}$ in
$\ck$;
\item[(iii)] $\{T^{k,n}\}$ satisfies condition $\frak{D}_1$ on a dense set $\frak{N}$ in
$\ck$.
\item[(iv)] for
each $k$ and every $x,y\in \ck$, there exists $\g_k\in[0,1)$ and
   $n_0=n_0(x,y,k)\in\N$ such that
\begin{equation}\label{L1-1}
\|T^{k,k+n_0}x-T^{k,k+n_0}y\|\leq \g_k\|x-y\|
\end{equation}
with
\begin{equation}\label{L1-2}
\limsup_{k\to\infty}\g_{n}<1;
\end{equation}
\item[(v)] $\{T^{k,n}\}$ is the $L$-weak ergodic;
\item[(vi)] $\{T^{k,n}\}$ satisfies condition $\frak{D}_1$ on $\ck$;
\end{enumerate}
the implications are true:
(i)$\Rightarrow$(ii)$\Rightarrow$(iii)$\Leftrightarrow$(iv)$\Leftrightarrow$(v)$\Leftrightarrow$(vi).
\end{thm}

\begin{proof} The implication (i)$\Rightarrow$ (ii) is obvious.
Consider (ii)$\Rightarrow$ (iii).  For a fixed $k\geq 0$, take
arbitrary $x,y\in \frak{N}$. Due to condition $\frak{D}_2$ one can
find $\l_k\in[0,2]$, $z_k\in D$ and two sequences
$\{\f_{k,x}^{(n)}\}$, $\{\f_{k,y}^{(n)}\}$ with
\begin{equation}\label{D22}
\|\f_{k,x}^{(n)}\|\to 0, \ \ \ \|\f_{k,y}^{(n)}\|\to 0 \ \
\textrm{as} \ \ n\to\infty
\end{equation}
 such that
\begin{equation}\label{D23}
T^{k,n}x+\f_{k,x}^{(n)}\geq\l_k z_k, \
T^{k,n}y+\f_{k,y}^{(n)}\geq\l_k z_k, \ \textrm{for all} \ \ n\geq k.
\end{equation}
Due to \eqref{D22} we choose $n_k$ such that
$$
\|\f_{k,x}^{(n_k)}\|\leq\frac{\l_k}{4}, \ \ \
\|\f_{k,y}^{(n_k)}\|\leq \frac{\l_k}{4}.
$$
Therefore, by denoting $\f_{k,x}=\f_{k,x}^{(n_k)}$,
$\f_{k,y}=\f_{k,y}^{(n_k)}$ from \eqref{D23} one finds
\begin{equation*}
T^{k,n_k}x+\f_{k,x}\geq\l_k z_k, \ T^{k,n_k}y+\f_{k,y}\geq\l_k z_k,
\end{equation*}
which yields condition $\frak{D}_1$ on $\frak{N}$.

(iii)$\Rightarrow$ (iv). Fix $k\in\N\cup\{0\}$, and take any two
elements $x,y\in\ck$. Then due Lemma \ref{3.2} one finds $u,v\in\ck$
such that
\begin{equation}\label{L1S12}
x-y=\frac{\|x-y\|}{2}(u-v).
\end{equation}

Since $\frak{N}$ is dense, for any $\i>0$ one can find
$u_1,v_1\in\frak{N}$ such that
\begin{equation}\label{L1S121}
\|u-u_1\|<\i, \ \ \ \|v-v_1\|<\i.
\end{equation}

According to condition $\frak{D}_1$, there exist $z_k\in D$ and
$\l_k\in[0,2]$ such that for those $u_1$ and $v_1$ one can find
$n_k\in\N$ and $\f_{k,u_1},\f_{k,v_1}\in X_+$ with
$\|\f_{k,u_1}\|\leq \frac{\l_k}{4}$, $\|\f_{k,v_1}\|\leq
\frac{\l_k}{4}$  one has
\begin{equation}\label{L1S13}
T^{k,n_k}u_1+\f_{k,u_1}\geq\l_k z_k, \ \
T^{k,n_k}v_1+\f_{k,v_1}\geq\l_k z_k.
\end{equation}

Now denote $\f_k=\f_{k,u_1}+\f_{k,v_1}$, then we have
\begin{equation}\label{L1S14}
\|\f_k\|\leq\frac{\l_k}{2}.
\end{equation}
From \eqref{L1S13} one finds
\begin{eqnarray}\label{L1S15}
&& T^{k,n_k}u_1+\f_{k}\geq\l_k z_k, \\[2mm]
\label{L1S16}
&&
T^{k,n_k}v_1+\f_{k}\geq \l_k z_k.
\end{eqnarray}

Therefore, using Markovianity of $T^{k,n}$, and inequalities
\eqref{L1S15},\eqref{L1S16} with \eqref{L1S14} we obtain
\begin{eqnarray*}
&&\|T^{k,n_k}u_1+\f_{k}-\l_k z_k\|=1-c_1,\ \ \
 \|T^{k,n_k}v_1+\f_{k}-\l_k
z_k\|=1-c_1,
\end{eqnarray*}
where $1-c_1\leq 1-\frac{\l_k}{2}$. Let us denote
\begin{eqnarray*}
u_2=\frac{1}{1-c_1}(T^{k,n_k}u_1+\f_{k}-\l_k z_k),\\[2mm]
v_2=\frac{1}{1-c_1}(T^{k,n_k}v_1+\f_{k}-\l_k z_k).
\end{eqnarray*}
It is clear that $u_2,v_2\in \ck$.

So, one has
\begin{equation}\label{L1S17}
T^{k,n_k}u_1-T^{k,n_k}v_1=(1-c_1)(u_2-v_2).
\end{equation}

Now from \eqref{L1S12} and \eqref{L1S17} we obtain
\begin{eqnarray*}
\|T^{k,n_k}x-T^{k,n_k}y\|&=&\frac{\|x-y\|}{2}\|T^{k,n_k}u-T^{k,n_k}v\|\\[2mm]
&\leq&\frac{\|x-y\|}{2}\big(\|T^{k,n_k}(u-u_1)\|+\|T^{k,n_k}(v-v_1)\|\\[2mm]
&&+\|T^{k,n_k}u_1-T^{k,n_k}v_1\|\big)\\[2mm]
&\leq&\frac{\|x-y\|}{2}\big(2\i+2(1-c_1)\big)\\[2mm]
&\leq&(\i+1-c_1)\|x-y\|\\[2mm]
&\leq & \bigg(\i+1-\frac{\l_k}{2}\bigg)\|x-y\|.
\end{eqnarray*}
Due to the arbitrariness of $\i$  and taking into account
\eqref{Lam1} we get the desired assertion.

(iv)$\Rightarrow$(v). Assume that $x,y\in \ck$ and $k\in\N\cup\{0\}$
are fixed. Due to condition (iv) one can find $\l_k\in[0,1)$ and
$n_0$ such that
\begin{equation}\label{L1-3}
\|T^{k,k+n_0}x-T^{k,k+n_0}y\|\leq \g_k\|x-y\|.
\end{equation}

 Now we claim that there are numbers
$\{n_i\}_{i=0}^m\subset\N$ and
\begin{eqnarray}\label{L1-3g}
\|T^{k,K_m}x-T^{k,K_m}y\|\leq\bigg(\prod_{j=0}^{m-1}\g_{K_j}\bigg)\|x-y\|,
\end{eqnarray}
where $K_0=k, K_{j+1}=k+\sum\limits_{i=0}^jn_i$, $j=0,\dots,m-2$.

Let us prove the inequality \eqref{L1-3g} by induction.

When $m=1$ we have already proved it. Assume that \eqref{L1-3g}
holds at $m$.

Denote $x_m:=T^{k,K_m}x, y_m:=T^{k,K_m}y$. It is clear that
$x_m,y_m\in\ck$. Then one can find $n_{m+1}\in\N$ and
$\g_{K_m}\in[0,1)$ such that
\begin{equation}\label{L1-3gm}
\|T^{K_m,K_m+n_{m+1}}x_m-T^{K_m,K_m+n_{m+1}}y_m\|\leq
\g_{K_m}\|x_m-y_m\|.
\end{equation}

Now using \eqref{L1-3gm} and our assumption one gets
\begin{eqnarray*}
\|T^{k,K_{m+1}}x-T^{k,K_{m+1}}y\|&=&
\big\|T^{K_m,K_{m+1}}(x_m-y_m)\big\|\\
&\leq& \g_{K_m}\|x_m-y_m\|\\
&\leq&\g_{K_m}\bigg(\prod_{j=0}^{m-1}\g_{K_j}\bigg)\|x-y\|.
\end{eqnarray*}

Hence, \eqref{L1-3g} is valid for all $m\in\N$. Take an arbitrary
$\i>$, then due to \eqref{L1-2} one can find $m\in\N$ such that
$\prod_{j=0}^{m}\g_{K_j}<\i$. Take any $n\geq K_m$, then we have
$$
n=K_m+r, \ \ 0\leq r<n_{m+1}
$$
hence from \eqref{L1-3g} one finds
\begin{eqnarray*}
\|T^{k,n}x-T^{k,n}y\|&=&\big\|T^{K_m,n}\big(T^{k,K_m}x-T^{k,K_m}y\big)\big\|\nonumber
\\[2mm]
&\leq &\|T^{k,K_m}x-T^{k,K_m}y\|<\i
\end{eqnarray*}
which implies the $L$-weak ergodicity.

(v)$\Rightarrow$(vi). Let $\{T^{k,n}\}$ be $L$-weak ergodic. Take
any $k\in\N\cup\{0\}$, and fix some element $y_0\in \ck$. Then for
any $x,y\in\ck$ from \eqref{L1wed} one gets
\begin{equation}\label{L1S18}
\|T^{k,n}x-T^{k,n}y_0\|\to 0, \ \ \|T^{k,n}y-T^{k,n}y_0\|\to 0\ \ \
\textrm{as} \ \ n\to\infty.
\end{equation}
Therefore, one can find $n_k\in\N$ such that
\begin{equation}\label{L1S19}
\|T^{k,n_k}x-T^{k,n_k}y_0\|\leq\frac{1}{4}, \ \
\|T^{k,n_k}y-T^{k,n_k}y_0\|\leq \frac{1}{4}.
\end{equation}

 Let us denote
$$
\f_{k,x}=(T^{k,n_k}x-T^{k,n_k}y_0)_-, \ \
\f_{k,y}=(T^{k,n_k}y-T^{k,n_k}y_0)_-,
$$
where
$T^{k,n_k}x-T^{k,n_k}y_0=(T^{k,n_k}x-T^{k,n_k}y_0)_+-(T^{k,n_k}x-T^{k,n_k}v_0)_-$.
From \eqref{L1S19} we obtain
$$
\|\f_{k,x}\|=\|(T^{k,n_k}x-T^{k,n_k}y_0)_-\|\leq
\|T^{k,n_k}x-T^{k,n_k}y_0\|\leq\frac{1}{4}.
$$
Similarly, one finds
$$
\|\f_{k,y}\|\leq\frac{1}{4}.
$$
It is clear that
\begin{eqnarray*}
T^{k,n_k}x+\f_{k,x}&=&T^{k,n_k}y_0+T^{k,n_k}x-T^{k,n_k}y_0+\f_{k,x}\\[2mm]
&=&T^{k,n_k}y_0+(T^{k,n_k}x-T^{k,n_k}y_0)_+\\
 &\geq& T^{k,n_k}y_0.
\end{eqnarray*}
Using the same argument, we have
\begin{eqnarray*}
T^{k,n_k}y+\f_{k,y}\geq T^{k,n_k}y_0.
\end{eqnarray*}
By denoting $\l_k=1$ and $z_k=T^{k,n_k}y_0$, we conclude that the
process $\{T^{k,m}\}$ satisfies condition $\frak{D}_1$ on $\ck$.

The implication (vi)$\Rightarrow$(iii) is obvious. This completes
the proof.
\end{proof}

\begin{rem} We stress that in \cite{SZ0,SZ} the $L$-weak ergodicity was established
for homogeneous Markov chain under a stronger condition than
$\frak{D_2}$ (i.e. the condition $\frak{D_2}$ with $\f_{k,x}\equiv
0$). Note that if $X$ is an $L^1$-space, then analogous theorem has
been proved in \cite{M2}. Moreover, certain concrete examples of
NDMC defined on $L^1$-spaces are provided.
\end{rem}

\begin{cor}\label{L1weq}
Let $(X,X_+,\ck,f)$ be an OBSB and $\{T^{k,n}\}$ be a NDMC on $X$.
Assume that $\{T^{k,n}\}$ satisfies condition $\frak{D}_1$ on $\ck$
with
\begin{equation*}
\a=\liminf_{k\to\infty}\l_{k}>0.
\end{equation*}
Then for each $k\in\N\cup\{0\}$ and $x,y\in\ck$ one can find
$N(k,x,y)\in\br_+$ such that
\begin{equation}\label{D1-est}
\|T^{k,n}x-T^{k,n}y\|\leq
C\bigg(1-\frac{\a}{2}\bigg)^{(n-k)/N(k,x,y)}\|x-y\|,
\end{equation}
where $C$ is a some constant.
\end{cor}

The proof immediately follows from the estimation \eqref{L1-3g}.

\begin{cor}\label{L1weq}
Let $(X,X_+,\ck,f)$ be an OBSB and $\{T^{k,n}\}$ be a NDMC on $X$.
If for each $k$ there exist $z_k\in \ck$ and a constant
$\l_k\in[0,2]$, and for every $x\in\ck$, one can find an
$\f_{k,x}\in X_+$ with $\|\f_{k,x}\|\leq \l_k/4$ such that
\begin{equation}\label{D1}
T^{k,k+1}x+\f_{k,x}\geq\l z_k
\end{equation}
with \eqref{Lam1}. Then $\{T^{k,n}\}$ is the $L$-weak ergodic.
\end{cor}

It turns out that the $L$-strong ergodicity implies condition
$\frak{D}_2$. Namely, one has

\begin{thm}\label{L1r1}
Let $(X,X_+,\ck,f)$ be an OBSB and $\{T^{k,n}\}$ be a NDMC on $X$.
If $\{T^{k,m}\}$ is the $L$-strong ergodic, then it
 satisfies condition $\frak{D}_2$ on $\ck$.
\end{thm}

\begin{proof} Take any $k\geq 0$ and fix arbitrary $x\in \ck$.
Then from the $L$-strong ergodicity one gets
\begin{equation}\label{L1r3}
\lim_{n\rightarrow\infty}\|T^{k,n}x-y_0\|=0,
\end{equation}
since $T_{y_0}x=y_0$.  Denote
\begin{equation*}
\f_{k,x}^{(n)}=(T^{k,n}x-y_0)_-.
\end{equation*}
From \eqref{L1r3} we obtain
$$
\|\f_{k,x}^{(n)}\|=\|(T^{k,n}x-y_0)_-\|\leq \|T^{k,n}x-{y_0}\|\to 0
\ \ \textrm{as} \ \ n\to\infty.
$$
It is clear that
$$
T^{k,n}x+\f_{k,x}^{(n)}=y_0+T^{k,n}x-y_0+\f_{k,x}^{(n)}=y_0+(T^{k,n}x-{y_0})_+\geq
y_0
$$
This implies that condition $\frak{D}_2$ is satisfied on $\ck$.
\end{proof}

\begin{rem} Note that in \cite{SG} it was proved that if
$\{T^{k,m}\}$ is a homogeneous Markov chain on a von Neumann
algebra, then condition $\frak{D}_2$ implies the $L$-strong
ergodicity, i.e. these two notions are equivalent. We should also
stress that this equivalence is not valid for general ordered Banach
spaces with a base (see \cite{SZ0}). \end{rem}

\section{Examples}

We first note that certain concrete examples of NDMC defined on
$L^1$-spaces are provided in \cite{M2}. Therefore, in this section
we are going to provide several examples of NDMC defined on OBSB
different from $L^1$, and satisfy the condition $\frak{D}_1$.

{\bf 1.} Let $X=C[0,1]$ be the space of real-valued continuous
functions on $[0,1]$. Denote
$$
X_+=\big\{x\in X: \ \max_{0\leq t\leq 1}|x(t)-x(1)|\leq 2
x(1)\big\}.
$$
Then $X_+$ is a generating cone for $X$, and $f(x)=x(1)$ is a
strictly positive linear functional. Then $\ck=\{x\in X_+: \
f(x)=1\}$ is a base corresponding to $f$. One can check that the
base norm $\|x\|$ is equivalent to the usual one
$\|x\|_{\infty}=\max\limits_{0\leq t\leq 1}|x(t)|$. Due to
closedness of $X_+$ we conclude that $(X,X_+,\ck,f)$ is a OBSB (see
\cite{SZ0}).

To define a NDMC $\{T^{k,m}\}$ is enough to provide a sequence of
Markov operators $\{T_k\}_{k=1}^\infty$ and in this case one has
$$
T^{k,m}=T_m\cdots T_k, \ \ k<m.
$$

For each $k\in\bn$, we now consider the following operator $T_k:X\to
X$ given by
$$(T_kx)(t)=t^kx(t).
$$
It is clear that $T_k$ is a
Markov operator on $X$ for every $k\in\bn$.

Let us show that the defined NDMC $\{T^{k,m}\}$ satisfies the
condition $\frak{D}_1$. Take any $k\in\bn$ and $x,y\in\ck$. Put
$\f_{k,x}\equiv 0$, $\l_k=1$ and $z_k=c$, $c\in(0,1/2)$. Then the
inequalities $T^{k,k+N}x\geq \l_k z_k$, $T^{k,k+N}y\geq \l_k z_k$
are equivalent to $T^{k,k+N}x-\l_k z_k,T^{k,k+N}y-\l_k z_k\in X_+$,
which mean
\begin{eqnarray*}
&&\max_{0\leq t\leq 1}|(T^{k,k+N}x)(t)-(T^{k,k+N}x)(1)|\leq
2\big((T^{k,k+N}x)(1)-z_k\big), \\[2mm]
&&\max_{0\leq t\leq 1}|(T^{k,k+N}y)(t)-(T^{k,k+N}y)(1)|\leq
2\big((T^{k,k+N}y)(1)-z_k\big).
\end{eqnarray*}
The last one can be rewritten as follows:
\begin{eqnarray*}
&&\max_{0\leq t\leq 1}|t^{(k+N/2)(N+1)}x(t)-x(1)|\leq 2(x(1)-c),
\\[2mm]
&&
 \max_{0\leq
t\leq 1}|t^{(k+N/2)(N+1)}y(t)-y(1)|\leq 2(y(1)-c).
\end{eqnarray*}
 Hence, taking
into account $x,y\in\ck$ one gets
\begin{eqnarray}\label{1cc}
&&\max_{0\leq t\leq 1}\big|t^{(k+N/2)(N+1)}x(t)-1\big|\leq
2(1-c),\\[2mm]\label{2cc}
&&\max_{0\leq t\leq 1}\big|t^{(k+N/2)(N+1)}y(t)-1\big|\leq 2(1-c) .
\end{eqnarray}
One can see that the existence of $N=N_k(x,y)$ such that
\eqref{1cc},\eqref{2cc} are satisfied. Thus, NDMC is satisfied the
condition $\frak{D}_1$, so due to Theorem \ref{L1S1} the defined
NDMC is $L$-weak ergodic.

Note that this chain is not $L$-strong ergodic. Moreover, it shows
that a main result of \cite[Theorem 2.1]{Ber} is false.

{\bf 2.} Let $L_p[0,1]$ ($1<p<\infty$) be the usual $L_p$-space. Let
$$X=\{(\a,{\mathbf{x}}): \ \a\in\br, \ {\mathbf{x}}\in
L_p[0,1]\}.$$ Define
$$X_+=\{(\a,{\mathbf{x}})\in X: \ \|{\mathbf{x}}\|_p\leq\a\}
$$
and $f(\a,{\mathbf{x}})=\a$. It is clear that $f$ is a strictly
positive linear functional on $X$. Put $\ck=\{(\a,{\mathbf{x}})\in
X_+: \ f(\a,{\mathbf{x}})=1\}$. One can check that $(X,X_+,\ck,f)$
is a OBSB. Now consider the following operator $T:X\to X$ defined by
\begin{equation}\label{tt1}
T(\a,{\mathbf{x}})=\bigg(\a,\a g(t)+\int K(s,t)x(s)ds\bigg), \ \
(\a,{\mathbf{x}})\in X.
\end{equation}
It is clear that $T$ is a linear operator. Let us find some
conditions to ensure its makovianity. Let $(\a,{\mathbf{x}})\in\ck$,
then $\a=1$ and $\|\mathbf{x}\|_p\leq 1$. From
$$
T(1,{\mathbf{x}})=\bigg(1,g(t)+\int K(s,t)x(s)ds\bigg).
$$
we conclude that the markoviaity is equivalent to
\begin{equation}\label{tt12}
\bigg\|g(t)+\int K(s,t)x(s)ds\bigg\|_p\leq 1.
\end{equation}
One can check that \begin{eqnarray}\label{tt2} \bigg\|g(t)+ \int
K(s,t)x(s)ds\bigg\|_p^p&=&\int\bigg|g(t)+\int K(s,t)x(s)ds\bigg|^p
dt\nonumber \\[2mm]
&\leq & \int\bigg(|g(t)|+\int |K(s,t)x(s)|ds\bigg)^p dt
\nonumber\\[2mm]
&\leq & \int\bigg(|g(t)|+\bigg(\int |K(s,t)|^q
ds\bigg)^{1/q}\|\mathbf{x}\|_p\bigg)^p dt
\nonumber\\[2mm]
&\leq & \int\bigg(|g(t)|+\bigg(\int |K(s,t)|^q
ds\bigg)^{1/q}\bigg)^p dt.
\end{eqnarray}
where $1/p+1/q=1$. Hence, the last inequality with \eqref{tt12}
yields that
\begin{equation}\label{t-Mar}
\int\bigg(|g(t)|+\bigg(\int |K(s,t)|^q ds\bigg)^{1/q}\bigg)^p dt\leq
1
\end{equation}
which implies the markovianity of $T$.

Now we choose $g(t)$ such a way that $z=(1,2g(t))\in\ck$ (this means
$\|g\|_p\leq 1/2$) and for an arbitrary $x=(1,\mathbf{x})\in\ck$ one
holds $Tx\geq \frac{1}{2}z$, i.e. from \eqref{tt1} the last
inequality is equivalent to
$$
\bigg(\frac{1}{2},\int K(s,t)x(s)ds\bigg)\geq 0.
$$
This means
$$
\bigg\|\int K(s,t)x(s)ds\bigg\|_p\leq\frac{1}{2}.
$$
It is clear that the last one is satisfied if one has
\begin{equation}\label{t-Mar1}
\int\bigg(\int |K(s,t)|^q ds\bigg)^{p/q} dt\leq
\bigg(\frac{1}{2}\bigg)^p.
\end{equation}

From the well-known inequality $(|a|+|b|)^p\leq
2^{p-1}(|a|^p+|b|^p)$ we immediately find that the inequalities
$\|g\|_p\leq 1/2$ and \eqref{t-Mar1} imply \eqref{t-Mar}.

 Hence, we conclude that if $\|g\|_p\leq 1/2$ and \eqref{t-Mar1}
are satisfied then the homogeneous Markov chain $\{T^n\}$ due to
Theorem \ref{N2} is weak ergodic.

Now we modify the provided construction to obtain NDMC. To do so, it
is enough to define a sequence $\{T_k\}$ of Markov operators. Let us
define  $T_k:X\to X$ as follows
\begin{equation}\label{tt1}
T_k(\a,{\mathbf{x}})=\bigg(\a,\a g_k(t)+\int H_k(s,t)x(s)ds\bigg).
\end{equation}

If we have
\begin{eqnarray}\label{t-Mar1-K}
&&\int\bigg(\int |H_k(s,t)|^q ds\bigg)^{p/q} dt\leq
\bigg(\frac{1}{2}\bigg)^p, \\[2mm] \label{t-Mar1-g}
&& \int|g_k(t)|^p dt\leq \bigg(\frac{1}{2}\bigg)^p, \ \  \ k\in\bn,
\end{eqnarray}
then for every $k\in\bn$ one find $z_k=(1,2g_k)\in\ck$ such that
$$
T_kx\geq \frac{1}{2}z_k, \ \ \textrm{for all}\ \ x\in\ck.
$$
This according to Corollary \ref{L1weq} implies that the NDMC
constructed by the sequence $\{T_k\}$ is $L$-weak ergodic.

Now let us consider a more concrete example. Take $p=2$ and
\begin{equation}\label{Hg}
H_k(s,t)=a_kt^{k/2}s^{k/2}, \ \ g_k(t)=b_kt^k.
\end{equation}
From
the estimation \eqref{t-Mar1-K} we find
$$
\int\int |a_k|^2t^k s^k ds dt\leq \frac{1}{4}
$$
which yields
\begin{equation}\label{aa}
|a_k|\leq \frac{k+1}{2}.
\end{equation}
Similarly, from \eqref{t-Mar1-g} we obtain
\begin{equation}\label{bb}
|b_k|\leq \frac{\sqrt{2k+1}}{2}.
\end{equation}
Consequently, if \eqref{aa},\eqref{bb} are satisfied then the NDMC
$\{T^{m,k}\}$ generated by the operators $T_k$ corresponding to
\eqref{Hg} is $L$-weak ergodic.

\section*{Acknowledgments} The author acknowledges the MOHE grant
ERGS13-024-0057 and the Junior Associate scheme of the Abdus Salam
International Centre for Theoretical Physics, Trieste, Italy.

%\newpage

\end{document}